\documentclass[letterpaper,11pt]{article}
\usepackage[letterpaper,margin=1in]{geometry}
\usepackage[T1]{fontenc}
\usepackage{graphicx}
\usepackage{amsmath,amssymb, amsthm, todonotes}
\usepackage{thmtools}
\usepackage{hyperref}
\usepackage{xcolor}
\usepackage{xspace}
\usepackage[most]{tcolorbox}
\usepackage[capitalise]{cleveref}
\usepackage{enumitem}
\newtheorem{theorem}{Theorem}
\newtheorem{conjecture}[theorem]{Conjecture}

\newtheorem{lemma}[theorem]{Lemma}
\newtheorem{corollary}[theorem]{Corollary}

\crefname{claim}{claim}{claims}
\Crefname{Claim}{Claim}{Claims}
\newcommand{\mydiamond}{\rotatebox[origin=c]{45}{$\vcenter{\hbox{$\Box$}}$}}

\newcommand{\say}[1]{``#1''}

\newcounter{tbox}

\title{Forbidding anticomplete planar minors: Induced Erd\H{o}s--P\'osa property and Maximum Independent Set in QP
\thanks{This work was initiated during the 'Focused Workshop on Erdős--Pósa problems' held in Będlewo, Poland, March 15-20, 2026, which was supported by the Banach Center of the Institute of Mathematics of the Polish Academy of Sciences.}
}

\author{Anonymous}
\author{Maria Chudnovsky \thanks{Princeton University, Princeton, NJ, USA. Supported by NSF Grants DMS-2348219 and CCF-2505100, AFOSR grant FA9550-25-1-0275, and a Guggenheim Fellowship.}
\and
Amadeus Reinald \thanks{University of Warsaw, Poland. Supported by Polish National Science Centre SONATA BIS-12 grant number 2022/46/E/ST6/00143.}
\and St\'{e}phan Thomass\'{e} \thanks{Univ. Lyon, ENS de Lyon, UCBL, CNRS, LIP, France. Supported by the ANR project GODASse (ANR-24-CE48-4377)}}

\begin{document}

\begin{titlepage}
\date{}
\maketitle

\begin{abstract}
The Erd\H{o}s--P\'osa theorem asserts that every graph $G$ with no $k$ disjoint cycles contains a set $X$ of $f(k)$ vertices such that $G\setminus X$ has no cycle. Robertson and Seymour showed that this \emph{Erd\H{o}s--P\'osa property} also holds for $H$-minor models of any planar graph $H$.
Equivalently, if $G$ has no $k$ minor models of $H$ pairwise at distance at least 1 (i.e. disjoint), then one can remove $f(k,H)$ balls of radius 0 (i.e. vertices) to make the graph $H$-minor free.

We show that this coarse graph theory point of view generalizes to distance at least 2 versus radius 1 balls, yielding the \emph{induced Erd\H{o}s--P\'osa property} for planar minors. Namely, every graph $G$ which does not contain $k$ pairwise non-adjacent minor models of a planar graph $H$ (we say that $G$ is \emph{$kH$-free}) can be made $H$-minor free by removing $f(k,H)$ neighborhoods. The proof  relies on the fact that sparse $kH$-free graphs have linearly many independent large protrusions. The same method gives that sparse $kH$-free graphs can be made $H$-minor free by deleting $O(\log n)$ vertices (and thus have logarithmic tree-width). This gives a quasi-polynomial algorithm for the Maximum Independent Set problem for $kH$-free graphs.
\end{abstract}

\thispagestyle{empty}
\end{titlepage}

\newpage\setcounter{page}{1}

\section{Introduction}

The Erd\H{o}s--P\'osa theorem~\cite{DBLP:journals/cjm/ErdosP65} asserts that every graph $G$ with no $k$ disjoint cycles contains a set $X$ of $f(k)$ vertices such that $G\setminus X$ is a forest.
In other words, the minimum feedback vertex set is functionally equivalent to the maximum packing of cycles.
This bridged duality gap is referred to as the \emph{Erd\H{o}s--P\'osa property} (EPP), and the term is now used for a vast collection of analogous results relating the packing and hitting number of more general objects.
It often has direct algorithmic applications.
For instance, Maximum Independent Set (MIS) is polynomial-time solvable in graphs without $k$ disjoint cycles, since their feedback vertex set, and hence their tree-width, is bounded.

To this day, variants of cycles have constituted a rich playing ground for the Erd\H{o}s--P\'osa property, in terms of both positive and negative results.
On the positive side, the EPP is known to hold for cycles with length constraints~\cite{Thomassen1988,MoussetNoeverSkoricWeissenberger2017}, holes~\cite{DBLP:journals/jct/KimK20}, cycles passing through prescribed vertices~\cite{KakimuraKawarabayashiMarx2011,PontecorviWollan2012}, and even oriented cycles~\cite{RRST96}.
On the negative side, odd cycles in an Escher Wall pairwise intersect yet their hitting set is unbounded~\cite{Ree99}, induced cycles of length at least $5$ do not have the EPP~\cite{DBLP:journals/jct/KimK20}, and neither do cycles of prime lengths~\cite{gorsky2026erd}.
If we wish to generalize the EPP to more complex objects, one can observe that cycles are $K_3$-minor models. It is then natural to ask if minor models of some fixed graph $H$ have the EPP. This is not always the case since toroidal grids do not contain two disjoint $K_5$-minor models, whereas the minimum hitting set of $K_5$-minor models has unbounded size.
Arguably, the most striking result in this field is that $H$-minor models have the EPP if and only if $H$ is planar.
This is a direct application of Robertson and Seymour's celebrated grid minor theorem~\cite{GMV}: either a graph has a large grid minor, and it contains $k$ disjoint $H$-minor models, or it has bounded tree-width, in which case a simple divide and conquer argument gives a hitting set of bounded size.

A recent thread in graph structure has been to generalize classical results to the \textsl{coarse} setting. For hitting versus packing problems, such analogues look at replacing disjointness conditions with large distance ones, and vertices with small radius balls.
In that direction, a natural coarse EPP for cycles was conjectured in~\cite{AGHK25}, and also by Chudnovsky and Seymour (see also~\cite{GP25}): either a graph $G$ contains a packing of $k$ cycles at pairwise distance at least $d$, or there exist $f(k)$ balls of radius $g(d)$ hitting all cycles of $G$. 
Strikingly, this was very recently confirmed by Dujmović, Joret, Micek and Morin~\cite{dujmovic2024erd}, who showed the above with $g(d)=19(d-1)$.
Aiming to push this result to more general structures, Albrechtsen and Davies~\cite{persoDaviesAlbrechtsen} recently conjectured the coarse EPP for planar minors (analogous to~\cite{GMV}): if $G$ does not contain $k$ minor models of $H$ at pairwise distance at least $d$, then $f(k)$ balls of radius $g(d)$ suffice to hit all $H$-minor models. As we will see, \Cref{thm:EP} confirms in a strong sense their conjecture for $d=2$.
A similar conjecture was made in~\cite{AGHK25} corresponding to the case $d=2$ above, but asking to pack \textsl{induced} minor models of a planar graph\footnote{Another positive hitting versus packing result is the recent proof of a coarse Gallai theorem~\cite{distel2026coarse}.}.
Unfortunately, coarse analogues of Menger's theorem, fundamental to much of classical graph structure, were recently ruled out, disproving a conjecture posed in~\cite{GP25} and~\cite{AHJKW24}.
Indeed, Nguyen, Scott and Seymour~\cite{NSS25a} (see also~\cite{NSS25}) showed that for any $d \geq 3, k \geq 3$, and for every $N$, there exist graphs with no $k$ $X,Y$-paths at pairwise distance at least $d$, but no $N$ balls of radius $N$ whose deletion disconnects $X$ from $Y$.

While the lack of a Menger analogue seems to predict a bleak future for coarse graph theory, the case $d=2$ (asking for non-adjacent paths) remains open even with balls of radius $1$. It appears that a wide range of classical results could still be generalized to this \textsl{induced setting}, relating objects constrained by non-adjacency with (closed) neighborhoods\footnote{It is often easy to construct counter-examples ruling out a relation between usual (vertex) hitting sets and induced structure, see~\cite{czyzewska2026induced}, and the first natural relaxation is to allow for neighborhoods.}.
For instance, Gartland and Lokshtanov~\cite{DBLP:phd/us/Gartland23} conjecture that for any planar $H$, graphs forbidding $H$ as an induced minor admit balanced separators consisting of few neighborhoods.
Abrishami, Czyżewska, Kluk, Pilipczuk, Pilipczuk, Rzążewski~\cite{abrishami2025coarsetreedecompositionscoarse} further conjecture that classes with such separators admit tree-decompositions whose bags can be covered by few bounded radius balls. The latter question is still open when asking for few neighbourhoods (see~\cite{chudnovsky2025coarsebalancedseparatorstreedecompositions}), which combined with a strengthening of~\cite{DBLP:phd/us/Gartland23} may yield an induced variant of the grid minor theorem. 
Coming back to packing versus hitting problems, the most pressing question in the induced setting is the induced Menger conjecture, formulated in~\cite{hickingbotham2025induced}, stating that graphs admitting no $k$ pairwise non-adjacent paths between two given vertices admit a set of $f(k)$ neighborhoods disconnecting them.

The \emph{induced Erd\H{o}s--P\'osa property} asks for a set of pairwise non-adjacent copies of an object, or a bounded number of neighborhoods hitting all copies (a \emph{dominating set} of the object). 
In this setting, the first proof of an induced EPP was obtained by Ahn, Gollin, Huynh and Kwon~\cite{AGHK25} for (non-induced) cycles of length at least $t$ with a dominating set of size $O(t k \log k)$.
This is particularly remarkable as the quasilinear dependence on $k$ matches the EPP for cycle packing (for fixed $t$).
In view of the fact that the classical EPP for cycles generalizes to planar minors~\cite{GMV} and~\cite{persoDaviesAlbrechtsen}, we ask the question for the induced EPP: do graphs forbidding $k$ independent copies of $H$ admit a hitting set of $H$-minor models consisting of $f(k)$ neighbourhoods?
%In particular, this question corresponds to the case $d=2$ of the coarse analogue conjectured in~\cite{persoDaviesAlbrechtsen}, strengthened to only allow hitting with radius one balls.

Our first result is that planar minors admit the induced Erd\H{o}s--P\'osa property.
We say a graph $G$ is \emph{$kH$-free} if it does not contain $k$ pairwise non-adjacent (\emph{independent}) $H$-minor models.
\begin{restatable}{theorem}{thmEP}\label{thm:EP}
    Given a planar graph $H$, there exists a function $f$ such that for every $kH$-free graph $G$, there exists $X \subseteq V(G)$ with $|X|\leq f(k)$ such that $G-N[X]$ is $H$-minor-free.
\end{restatable}
\noindent
To the best of our knowledge, Theorem~\ref{thm:EP} is one of the few results on the EPP for such a general family of objects. Apart from planar minors admitting the EPP~\cite{GMV}, Liu~\cite{Liu22} showed that \textsl{all} minors admit the \say{half-integral} EPP, and Amiri, Kawarabayashi, Kreutzer and Wollan~\cite{amiri2016erdosposapropertydirectedgraphs} showed the EPP for butterfly (or topological) minors of the cylindrical grid in digraphs.
In particular, \Cref{thm:EP} confirms the planar coarse EPP conjecture of~\cite{persoDaviesAlbrechtsen} for $d=2$ in a strong form.
Let us also note that~\cite{czyzewska2026induced} very recently considered the induced EPP for induced \textsl{objects}, showing the induced EPP for long holes and long theta induced minors. They conjecture more generally that for any planar $H$, induced minor models of $H$ satisfy the induced EPP (a strenghtening of~\cite{AGHK25}), which corresponds to replacing \say{minor} with \say{induced minor} in~\Cref{thm:EP}.

The obstacles to coarse graph theory are not the only motivation to understand the induced setting.
On the algorithmic side, the tractability of hard problems in classes forbidding some induced substructure has a long history.
The most studied problem there is Maximum Independent Set, both for exact algorithms and approximations~\cite{chudnovsky2020quasi,bonnet2026qptasmwisfindinglarge,pilipczukQuasiPoly,10.1145/3414473}, with the most general conjecture that planar $H$-induced minor free graphs admit a polynomial-time algorithm for MIS~\cite{DBLP:phd/us/Gartland23}.
When non-adjacency is imposed \textsl{between} objects, as for $kH$-freeness, some geometric problems naturally involve instances that forbid independent substructures.
For example, given as input some finite $X\subseteq \mathbf{R}^3$, the computation of a maximum size subset of $X$ whose elements are pairwise at distance less than 1 admits an EPTAS, see~\cite{BBBCGKRST21} and~\cite{BBBCT18}.
The algorithm consists of computing the Maximum Independent Set in the distance at least one graph, and crucially relies on the fact that such graphs forbid two independent odd cycles.
In this vein, one of the basic open questions is the computation of MIS in polynomial time for graphs without $k$ non-adjacent cycles (i.e. $kK_3$-free graphs). A quasi-polynomial time (QP) algorithm was proposed in \cite{BBDEGHTW23}.
%A wide array of results linking the tractability of MIS with odd cycle packing can be found in \cite{AWZ17}, \cite{BFMR14}, \cite{CFHJW20}, and \cite{FJWY21}.

Our main algorithmic result generalizes the quasi-polynomiality of MIS to graphs forbidding $k$ independent planar minors.
\begin{restatable}{theorem}{thmMIS}\label{thm:MISQP}
    Let $H$ be a planar graph and $k$ be an integer. There is an $n^{O(\log n)}$ time algorithm computing MIS in $kH$-free graphs.
\end{restatable}

The MIS algorithm for $kK_3$-free graphs in~\cite{BBDEGHTW23} is based on the fact that sparse $kK_3$-free graphs (i.e. avoiding some fixed $K_{r,r}$ subgraph) have logarithmic feedback vertex set (and thus logarithmic tree-width). Therefore MIS can be computed by dynamic programming in polynomial time. Unfortunately, reducing to the sparse case incurs a QP branching preprocess. Our QP algorithm for MIS in $kH$-free graphs follows exactly these lines. We start with a reduction to the sparse case via QP branching, and then apply our second main result: sparse $kH$-free graphs have logarithmic $H$-hitting set.

\begin{restatable}{theorem}{thmLOG}\label{thm:EP-sparse}
    Given a planar graph $H$, there exists a function $f$ such that for every $kH$-free graph $G$ on $n$ vertices with no $K_{r,r}$ subgraph, $G$ admits an $H$-hitting set of size at most $f(k,r) \log n$.
\end{restatable}
\noindent
Since removing an $H$-hitting set leaves a graph with bounded tree-width, this result directly implies that sparse $kH$-free graphs have logarithmic tree-width. For more classes having the logarithmic tree-width property, see \cite{ACHS22}.

Regarding~\Cref{thm:MISQP}, let us highlight one of the tantalizing questions of this domain, which is a weakening of~\cite[Conjecture 1.4.2]{DBLP:phd/us/Gartland23}, and remains open in the $2K_3$-free case.
\begin{conjecture}
    For any planar $H$ and any $k$, the Maximum Independent Set problem is polynomial-time solvable in the class of $kH$-free graphs.
\end{conjecture}

Another interesting algorithmic question regarding $kH$-free graphs is that of recognition. There, a polynomial time algorithm for $H=K_3$ was shown by Nguyen, Scott and Seymour~\cite{NGU24}.

\paragraph{Structure of the paper}
We end this introduction with overviews of~\Cref{thm:EP}, \Cref{thm:EP-sparse} and~\Cref{thm:MISQP}, all based on the \emph{linear protrusion property} (\Cref{thm:LPPkH}), which is the cornerstone of our results.
In Section~\ref{sec:prelim} we introduce the main definitions and provide a deeper overview of protrusion reduction and the linear protrusion property.
In Section~\ref{sec:realprot}, we show how protrusions can be reduced. In Section~\ref{sec:LPP}, we introduce our main tool, the linear protrusion property. In Section~\ref{sec:EPP}, we show the induced version of the EPP for planar minors. In Section~\ref{sec:MIS}, we give a QP algorithm for MIS in $kH$-free graphs, based on the logarithmic $H$-hitting set property.

\subsection{Overview of the induced Erd\H{o}s--P\'osa property for planar minors}\label{ssec:overview-EPP}

For the sake of exposition, we assume here that $H$ is connected, and show how to derive the disconnected case in~\Cref{sec:EPP}.
A natural strategy to show that a (huge) bounded number of neighborhoods hit all $H$-minor models in a $kH$-free graph is to consider a minimum counterexample $G$, and argue that some reduction can be performed.
Our first step is to reduce to the sparse case. Say that the \emph{$H$-girth} of $G$ is the size of the smallest $H$-minor model in $G$ (the $K_3$-girth then corresponds to the usual girth).
Note that if there is a small $H$-minor model $\mathcal{H}$ in $G$, the graph $G\setminus N[\mathcal{H}]$ is $(k-1)H$-free, thus $G$ is not a minimum counterexample (since we can choose a huge bound). Thus we can assume that $G$ has arbitrarily high $H$-girth.

An important remark is that there are sparse $kH$-free graphs with unbounded tree-width. Examples of high girth and high tree-width graphs with no two independent cycles (later on called \emph{death stars} for unknown reasons) can be found in~\cite{BBDEGHTW23}. So $G$ can still be rather complex, and we need to reduce it further. Isolated or degree one vertices can be deleted, so we can assume that the minimum degree is 2.
Then, a natural toy-problem is the following: how to reduce a $kK_3$-free graph with high girth? 

It turns out that such graphs have long proper paths, i.e. with internal vertices of degree 2 (see technical overview in Section~\ref{sec:lin}).
This shows Theorem~\ref{thm:EP} for $kK_3$-free graphs, since one can contract an edge in a long proper  path, contradicting the minimality of the counterexample. More generally, every $kH$-free graph with large $H$-girth has a $p$-vertex cutset $A$ partitioning $V(G)\setminus A$ into two arbitrarily large sides $V_1,V_2$
(for some fixed $p$, see Section~\ref{sec:lin}). Note that if both $V_1$ and $V_2$ contain an $H$-minor model, then they are both $(k-1)H$-free, contradicting the minimality of the counterexample (if the bound is chosen larger than twice the bound for the $(k-1)H$-free case plus $p$). Thus we can assume that $G[V_1]$ is arbitrarily large, $H$-minor free (thus of bounded tree-width) and has boundary size at most $p$. This is the definition of a $p$-protrusion $P$, with boundary $A$, a very popular tool in FPT algorithms, see~\cite{FL89}, \cite{FLMS12}, and  \cite{KLPRRSS15}. The key-property of large $p$-protrusions is that they can be reduced while preserving some properties of $G$. Therefore, similarly to the contraction of a proper path in the $kK_3$-case, the protrusion $P$ can be reduced, thus contradicting the minimality of $G$. The full proof of existence of large $p$-protrusions can be found in Section~\ref{sec:LPP}, Theorem~\ref{thm:LPPkH}.

Let us hint at how to reduce the $p$-protrusion $G[V_1]$. The crucial observation is that no minimum size set $X$ such that $N[X]$ hits all $H$-minor models intersects $V_1$ on more than $p$ vertices, otherwise swapping $X\cap V_1$ for the boundary of $V_1$ would be a smaller hitting set. So we can encode all the possible types of intersections of $G[V_1]$ and the optimal solution in a bounded way. Therefore if $V_1$ is too large, it can be replaced by a smaller graph. The technical overview of this protrusion reduction argument is discussed in Section~\ref{sec:prot}, and then proved in Section~\ref{sec:realprot}. 

\subsection{Overview of the logarithmic hitting set, and of the QP algorithm for MIS}\label{ssec:overview-MIS}

To compute MIS, we follow our usual recipe and preprocess our input graph $G$ to reduce it to (not too many) sparse instances.
Let us illustrate this for $2H$-free graphs. Since a small $H$-minor model $\mathcal{H}$ (say of size at most $s$) is adjacent to all other $H$-minor models, there is a vertex $v$ of $\mathcal{H}$ which is a neighbor of at least a fraction of $1/s$ of the number of small $H$-minor models. Thus, deleting $N[v]$ reduces the total number of small $H$-minor models by a factor $1-1/s$ (and there are at most $n^s$ of them). Therefore, choosing $v$ as a branching vertex to compute MIS leads to a recursion tree with a quasi-polynomial number of leaves, each leaf corresponding to a graph without a small $H$-minor model. We now need to solve MIS for sparse $kH$-free graphs.

The key is Theorem~\ref{thm:EP-sparse} which shows that sparse $kH$-free graphs have a logarithmic size hitting set $X$ of all $H$-minor models. With this in hand, we just run over all possible sets $X$ and all possible intersections $I$ of $X$ with a MIS, and extend $I$ to $G\setminus X$ (which has bounded tree-width since $H$ is planar) using dynamic programming. This again takes quasi-polynomial time. 
Therefore the main difficulty is to prove  Theorem~\ref{thm:EP-sparse}. The first step is to observe that the proof of the existence of a large $p$-protrusion in $kH$-free graphs with large $H$-girth gives in fact a linear number of independent large $p$-protrusions, see Section~\ref{sec:lin} for a technical overview and Theorem~\ref{thm:LPPkH}. 

The minimum size $L$ of our $p$-protrusions is set in such a way that each of them can be reduced to a smaller one while both preserving $kH$-freeness and the size of a minimum $H$-hitting set (see the overview in~Section~\ref{sec:prot}).
We can then reduce all at once this linear set of protrusions, in order to form a new graph $G'$ which has same size of minimum $H$-hitting set. Note that the number of vertices of $G'$ is a fraction of $|V(G)|$ since we strictly reduced a linear number of sets. The problem is that the $H$-girth of $G'$ has decreased, and we cannot apply 
Theorem~\ref{thm:LPPkH} again.
 Crucially, note that we are not in the process of \emph{computing} an $H$-hitting set of $G$, but just showing that one of logarithmic size exists. We can then focus on $G'$ and forget $G$.

The engine of our algorithm is a sparsification step, based on the fact that $G'$ cannot become too dense. In particular, if the value of $p$ is chosen (among other properties) such that $K_{p,p}$ contains an $H$-minor model, the graph $G$ does not contain $K_{2p,2p}$ (since this would be a small $H$-minor model). Thus our parallel $p$-protrusion reduction from $G$ to $G'$ also leaves $G'$ without $K_{2p,2p}$ subgraphs. Indeed, $K_{2p,2p}$ cannot cross the boundary of a $p$-protrusion, and it cannot live inside a $p$-protrusion since protrusion reduction respects the (non) existence of  $K_{2p,2p}$ subgraphs. Therefore, while we may not control the $H$-girth of $G'$, at least $G'$ is still somewhat sparse.

In the sparsification step (which is very classical in this domain), we delete a bounded number of vertices to reduce $G'$ to a new graph with no small $H$-minor model. If we can indeed achieve this, and repeat the process, we will iteratively delete $O(\log n)$ times a bounded set of vertices, showing that our starting graph $G$ has a logarithmic size $H$-hitting set. 
To finish the argument, note that if a $kH$-free graph has no $K_{2p,2p}$ subgraph, then one can delete $f(k,\ell,p)$ vertices to suppress all small $H$-minor models. It suffices to observe that a maximum size $N$ packing of small $H$-minor models induces a graph with $\Omega(N^2)$ edges and $O(N)$ vertices: every subset of $k$ disjoint small $H$-minor models contains an edge joining two of them. Such a quadratic density gives a $K_{2p,2p}$ if $N$ is large enough. So $N$ is bounded, and therefore deleting at most $N$ small $H$-minor models  leaves $G'$ with large $H$-girth.

\paragraph{Concluding remarks.}
Our proofs virtually never consider the structure of $H$, barely that of $G$, and do not effectively construct any hitting or dominating set. This is graph theory without graphs, just consisting of parameter tuning.
It is quite surprising to obtain some results and  algorithms without any insight on the fine-grained structure of $kH$-free graphs, apart from the crucial linear protrusion property. Compared to the original EPP of planar graphs, in which tree-width fully explains the phenomenon, this study does not (yet?) reveal any new complexity or width parameter hidden behind induced packing of planar graphs. Still, the linear protrusion property can be a useful tool for other graph classes.

\section{Preliminaries and overview of the linear protrusion property}\label{sec:prelim}

\subsection{Preliminaries}

We generally assume in this paper that $G=(V,E)$ is a graph on $n$ vertices. When $X$ is a subset of vertices of $G$, we denote by $G[X]$ the subgraph of $G$ \emph{induced} by $X$. We denote by $N(v)$ the neighborhood of a vertex $v$, and by $N[v]$ the closed neighborhood $N(v)\cup \{v\}$. The neighborhood of a set of vertices $X$ is $N(X):=(\bigcup_{x\in X} N(x))\setminus X$ and its closed neighborhood is 
$N[X]:=\bigcup_{x\in X} N[x]$.

Two disjoint sets of vertices $X,Y$ in a graph $G$ are \emph{independent} if there is no edge $xy$ such that $x\in X$ and $y\in Y$. We also call \emph{independent} a family $\mathcal{F}$ of pairwise independent subsets. When all subsets in $\mathcal{F}$ are singletons, we speak of an \emph{independent set} of $G$. It is a \emph{maximum independent set} (MIS) of $G$ if it has maximum size. 

A \textsl{class} $\mathcal C$ of graphs is always supposed to be closed by taking induced subgraphs. A class $\mathcal C$ is \emph{sparse} if there exists $r$ such that no graph in $\cal C$ contains $K_{r,r}$ as a subgraph. A class $\mathcal C$ is \emph{$\delta$-degenerate} if every graph in $\cal C$ on $n$ vertices has at most $\delta n$ edges. A class $\mathcal C$ has \emph{linear neighborhood complexity} with parameter $c$ if for every graph $G$ in $\cal C$ and for every non empty subset $X$ of vertices of $G$, the size of $\{N(v)\cap X:v\in V(G)\}$ is at most $c.|X|$. In other words, the number of distinct neighborhoods of vertices of $G$ in $X$ is linear in the size of $X$.

When $H$ is some graph, an \emph{$H$-minor model} in $G$ is a set $\mathcal{H}$ of $h:=|V(H)|$ disjoint connected subsets $\{X_1,\dots ,X_h\}$ of $G$ such that $G[\cup \mathcal{H}]/\mathcal{H}$ contains $H$ as a (spanning) subgraph. When $G[\cup \mathcal{H}]/\mathcal{H}$ is isomorphic to $H$, we speak of \emph{induced minor model}. A graph $G$ is \emph{$H$-minor-free} if it does not contain $H$ as a minor. By extension, a 
class $\mathcal{C}$ is 
\emph{$H$-minor free} if all graphs in $\mathcal{C}$ are $H$-minor free.
In a graph $G$, an \emph{$H$-hitting set} is a set $X$ of vertices such that $G\setminus X$ is $H$-minor free. We denote by $\tau_H(G)$ the minimum size of an $H$-hitting set. An \emph{$H$-dominating set} is a set $X$ of vertices such that $G\setminus N[X]$ is $H$-minor free. We denote by $\gamma_H(G)$ the minimum size of an $H$-dominating set.

Given some graph $H$ and an integer $k$, a graph $G$ is \emph{$kH$-free} if for any set $\mathcal{H}_1,\dots ,\mathcal{H}_k$ of disjoint $H$-minor models in $G$, there are some $i\neq j$ such that there exists an edge between $\mathcal{H}_i$ and $\mathcal{H}_j$. We will also say that there are no $k$ disjoint \emph{independent} $H$-minor models in $G$. By extension, a class $\mathcal C$ is \emph{$kH$-free} if  all $G$ in $\mathcal C$ are $kH$-free.

\begin{theorem}\label{thm:RSgrid}
     For every planar graph $H$, $H$-minor free graphs have bounded tree-width. 
\end{theorem}

\begin{proof}
    This follows from Robertson and Seymour grid minor theorem~\cite{GMV}.
\end{proof}

\begin{theorem}\label{thm:DVO}
    Every sparse $kH$-free class $\cal C$ is degenerate and has linear neighborhood complexity. 
\end{theorem}

\begin{proof}
Dvo\v{r}\'ak~\cite{Dvorak} showed that every sparse class not inducing some subdivision of a fixed graph $F$ has bounded expansion, therefore it is degenerate, and moreover, it has linear neighborhood complexity, as shown by Gajarský et al.~\cite{GAJARSKY2017219}. The conclusion follows since $\cal C$ does not induce a subdivision of the disjoint union of $k$  independent copies of $H$.
\end{proof}

\begin{theorem}\label{thm:KOR}
    There exists a function $g$ such that $\tau_H(G)\leq g(k,d)$ for every $kH$-free graph $G$ with maximum degree $d$.
\end{theorem}

\begin{proof}
    Korhonen~\cite{KORHONEN2023206} showed that there exists $f$ such that every graph with maximum degree $d$ and tree-width at least $f(N,d)$ contains a $N\times N$ grid as an induced minor. To show that $g$ exists, we set $N$ such that every  $N\times N$ grid induced minor contains $k$ independent $H$-minor models.
    It follows that $G$ has a tree-decomposition $T$ with width less than $f(N,d)$.
    When $H$ is connected, $T$ has a bag $B$ such that every component of $G \setminus B$ is $(k-1)H$-free. This follows by the standard argument orienting the edges of $T$ towards the side that contains a $(k-1)H$ and picking a sink. Then, we get $g(k,d)=(k-1)f(N,d)$ by induction on $k$.
    
    We now deal with the case where $H$ has connected components $H_1,\dots ,H_C$ with $C \geq 2$, and assume that we have shown that a function $g(k,d)$ exists for any $H$ with less than $C$ components.
    We can assume that $G$ contains $f(N,d)kC$ disjoint $H_1$ minors, otherwise we conclude by the connected case.
    Note that if a connected component $G_i$ of $G$ contains at least $f(N,d)+1$ copies of $H_1$, there is a separator $S$ of size at most  $f(N,d)$ such that $G_i \setminus S$ contains at least two components with a copy of $H_1$.
    Removing $S$ from $G$, and iterating this process increases the number of connected components of $G$ containing $H_1$ by at least one, and decreases the maximum number of disjoint copies of $H_1$ by at most $f(N,d)$. In turn, this process can be repeated for $kC$ iterations, such that the resulting $G$ contains at least $kC$ connected components each containing a copy of $H_1$, call those \emph{good}. 
    Now, if $G$ contains $k$ independent copies of $H\setminus H_1$, their union intersects at most $k(C-1)$ good components. Therefore, since there are at least $k$ good components which are not intersected, we can add $k$ independent copies of $H_1$ to form $k$ independent copies of $H$.
    Thus $G$ is $k(H\setminus H_1)$-free and we can conclude by induction.
\end{proof}

We now introduce some parameters related to our planar graph $H$. We set these (global) variables for the rest of the paper. 
\begin{itemize}
    \item $h$ is the number of vertices of $H$ and  $m$ is the maximum between $h$ and $|E(H)|$. Note then that $H$ is a minor of the complete bipartite graph $K_{m,m}$.
    \item  $t$ is the maximum tree-width of an $H$-minor free graph, and $p:=m+2t$ will be the boundary size of our protrusions.
    \item By Theorem~\ref{thm:DVO}, $kH$-free graphs $G$ with no $K_{m,m}$ subgraph have linear neighborhood complexity. Let $c$ be such that $|\{N(v)\cap S:v\in V\}|\leq c.|S|$ for every non empty $S\subseteq V(G)$. 
    \item By Theorem~\ref{thm:DVO}, there exists $\delta$ such that every $kH$-free graphs with no $K_{m,m}$ subgraph are $\delta$-degenerate, thus have average degree at most $2\delta$.

\end{itemize}

\subsection{Protrusion reduction}\label{sec:prot}

We formally define protrusion reduction in Section~\ref{sec:realprot}.
This subsection is just a gentle introduction to the main ideas since the protrusion literature can be a little bit intimidating and technical. A nice introduction to the topic (in the case of topological minors) can be found in~\cite{basteHittingMinorsProtrusion}, see also~\cite{bodlaender2016metaKernelization} for kernelization aspects.
Fortunately, the main point is based on a very simple idea, which is basically the pigeonhole principle. The difficulty is to express the fact that the number of different holes for the pigeons is bounded. This is generally done using MSOL-formulas up to some bounded depth, but we will try to make this paper both $kH$-free and MSOL-free.

A $p$-protrusion $P=(A,B)$ in a graph $G$ is such that $|A|\leq p$, $N(B)\subseteq A$ and $G[B]$ is \say{simple}, which generally means bounded tree-width. The tree-width requirement will be satisfied in our case, but irrelevant for our proofs, so let us simply forget it. When we say that $P$ contains a minor, we generally mean that $B$ contains it.
For instance, isolated vertices are 0-protrusions, degree 1 vertices are 1-protrusions, and a path $P$ consisting of degree 2 internal vertices (called \emph{proper path}) is a 2-protrusion. All these objects are generally treated (both in graph theory and algorithms) in the same way: we delete or contract them to find an equivalent instance of the problem. 
Finding a reduction becomes much less clear when a very large set $B$ is isolated by, say, 10 vertices. Fortunately, in our minimum counterexample strategy, we just focus on the existence of a reduction, so we can drop the (tedious) question of \emph{how} to find them efficiently. Let us focus on why large protrusions can be reduced.

The key-question is:  does the underlying problem depend (or not) on a bounded size set of properties of $G[A\cup B]$? If this is the case, then the total number of different types of $G[A\cup B]$ (with respect to these properties) is at most some $N$. In particular, for every type $i$, if we let $m_i$ be the minimum size of a $p$-protrusion of type $i$ and let $L$ be the maximum of $m_i+1$, where $i=1,\dots ,N$, then every $p$-protrusion of size at least $L$ can be reduced to a smaller equivalent one.

%In the particular context of $H$-hitting set in $kH$-free graphs, note that if a $p$-cutset separates $G$ into two large parts $B,B'$, then one of them, say $B$, is $H$-minor free, otherwise we can reduce to the $(k-1)H$-free case. The key-observation is that no minimum hitting set intersects $B$ on more than $p$ vertices, since $N[B]$ would be a better choice. Therefore, if such a large protrusion $B$ exists, we would be able to reduce $G$ if we could encode all relevant properties of $B$ in a bounded way. Namely, we want to reduce $B$ so that the new graph is still $kH$-free and has the same size minimum hitting set. 

Robertson and Seymour introduced in~\cite{ROBERTSON199565} the concept of folio in the context of minors, which we generalize here for induced minors.
Let us call $f$-folio of $P$ all the induced minors of size at most $f$ one can find in $G[A\cup B]$, with the additional information of the respective intersections of the connected parts of the minor models with $A$.
The main point to check (and a good warm-up) is that if we replace a $p$-protrusion $B$ by a $p$-protrusion with the same $(p+f)$-folio to form a new graph $G'$, then $G$ and $G'$ have the same induced minors on at most $f$ vertices. 

Assuming for the sake of exposition that $H$ is connected, recall (\Cref{ssec:overview-EPP}) that a $H$-minor free $p$-protrusion $B$ cannot intersect a minimum $H$-hitting set on more than $p$ elements.
We now define the $(s,f)$-folio set as the set of all $f$-folios 
one can obtain in $G[(A\cup B)\setminus X]$ by deleting a set $X$ of size at most $s$ with the additional information $A\cap X$. Observe that the parameter $\tau_H$ is invariant when we replace an $H$-minor free $p$-protrusion by one with the same $(p,h+p)$-folio sets.

To get the same invariance for the parameter $\gamma_H$, we now need to encode all $f$-folios one can obtain by deleting the neighborhood of a set $X$ of size at most $s$. We need then to specify all possible intersections $X\cap A$ and $N(X)\cap A$. This is still a bounded size information which we call $(s,f)$-neighborhood folio (the exact definition of neighborhood folio is slightly more technical, see \Cref{sec:realprot}). 
Again, the parameter $\gamma_H$ is invariant when replacing an $H$-minor free $p$-protrusion by one with the same neighborhood folio (with suitable parameters).  

Therefore, bounding $\gamma_H$ for all $kH$-free graphs $G$ reduces to show that there exists a large $p$-protrusion. 
%Indeed, if a $p$-vertex cut separates $G$ into two arbitrarily large parts $B,B'$, one of these (say $B$) is $H$-minor free (otherwise we apply induction). Thus $B$ is a large $H$-minor free $p$-protrusion. Since the number of distinct neighborhood folios (with some fixed suitable parameters) is bounded, if $B$ is larger than some bound $L$ (only depending on $p,k$ and $H$) then it can be reduced. In particular, $\gamma_H$ cannot be arbitrarily large.
We now turn to the central tool of this paper, how to find a linear number of large protrusions. 

\subsection{Linear Protrusion Property}\label{sec:lin}

We illustrate how our method works in the cycle case, and discuss afterward its generalization to any planar $H$. The goal is to overview the proof of Theorem~\ref{thm:LPPkH}, which asserts that for every $L$, there is a bound $g_L$  such that whenever a $kH$-free graph has $H$-girth more than $g_L$, it has a linear number of independent $p$-protrusions of size at least $L$.

Assume now that $G$ is $kK_3$-free with minimum degree 2 (vertices with degree at most 1 can be easily treated separately). Our goal is to show that if $G$ has (high) girth $g$, then it contains a linear number of long proper paths.

 The trick is to pick some very large $d$ and partition $V(G)$ into the vertices $Y$ of degree at most $d$ and the vertices $X$ of degree more than $d$. Since $G$ is degenerate (large girth), the ratio $|Y|/|X|$ is arbitrarily close to 1. Thus, the graph $G[Y]$ is $kK_3$-free with degree at most $d$ so by Theorem~\ref{thm:KOR}, it satisfies $\tau_{K_3}(G[Y])\leq C$ for some constant $C$. Hence there is a subset $Z$ of size at most $C$ inside $Y$ such that $G[Y\setminus Z]$ is $K_3$-minor free, thus is a forest. Observe that if the number of vertices of $G$ is much larger than $C$, we still have $|Y\setminus Z|/|X\cup Z|$ arbitrarily close to 1. Rename now $X\cup Z$ as $X$ and $Y\setminus Z$ as $Y$.
Note that every leaf in $G[Y]$ has a neighbor in $X$ since the minimum degree is 2.

Unfortunately, due to parameter tuning, the full proof for cycles is not much simpler than the general $kH$-free case. However, since $G[Y]$ is a forest, the argument is a little bit simpler to present. We discuss three possible examples for $G[Y]$ illustrating the main parts of the proof of Theorem~\ref{thm:LPPkH}. These three cases correspond to a trichotomy lemma for trees. Either a tree $T$ on $N$ nodes has $o(N)$ leaves and then almost all vertices have degree 2, thus there are linearly many long proper paths (our first example).
Or $T$ has $\Omega(N)$ leaves in which case either it has linearly many vertex-disjoint leaf to leaf paths (our second example), or there is a $o(N)$ subset of nodes of $T$ whose deletion leaves connected components containing at most one leaf of $T$ (our last example).

\begin{enumerate}

\item $G[Y]$ has linearly many (much more than $|X|$) long proper paths $(P_i)$. Half of them have bounded size (since they are linearly many). If half of these latter have at least two neighbors on $X$, we obtain many (still much more than $|X|$) short $(X,X)$-paths $(P'_i)$. Contracting these paths does not create $C_4$ since the girth is very high, so the linear neighborhood complexity property still holds. But then two of these $P'_i$'s must share endvertices on $X$, contradicting the high girth hypothesis. So a subfamily $(P''_i)$ consisting of a quarter of the family $(P_i)$ has at most one neighbor on $X$. Since   at most one edge connects $P''_i$ to $X$ due to the girth hypothesis, each $P''_i$ path contains a proper path of $G$ of size at least $|P''_i|/2$.

    \item $G[Y]$ has many (much more than $|X|$ and $n/g$) disjoint connected subsets $(S_i)$ each with at least two leaves. This would happen for instance in a binary tree. Then half of these sets have size much less than $g$. Hence contracting all $S_i$ gives a $C_4$-free graph $G'$ which is still $kK_3$-free. The contradiction appears since $G'$ has linear neighborhood complexity and there are too many $S_i$, each with degree at least 2 on $X$.
    
    \item $G[Y]$ has few internal vertices and many leaves. This happens for instance in a caterpillar of length $\sqrt{n}$ with internal vertices of degree $\sqrt{n}$. To solve this case, we have to change our point of view and consider the set $S$ of leaves of $G[Y]$ versus the rest of the graph. Since each leaf has another neighbor in $X$, and the size of $S$ is almost $n$, we contradict the linear neighborhood complexity on $V\setminus S$, or $C_4$-freeness.

\end{enumerate}

The full proof of Theorem~\ref{thm:LPPkH} follows the lines of the cycle case. The first step is to show that almost all of $G$ consists in an induced subgraph $G[Y]$ with tree-width at most $t$. Then Theorem~\ref{thm:SLPPTW} gives a linear set of large independent $2t$-protrusions in $G[Y]$ (this is the previous trichotomy result applied to the decomposition tree). Finally, as for the path case discussed for cycles, a fraction of these protrusions provides a linear set of large $p$-protrusions of $G$. 

To sum-up, a trichotomy argument shows that graphs with  tree-width $t$ have a linear number of large independent $2t$-protrusions. Then, since a $kH$-free graph $G$ with large $H$-girth almost entirely consists in a graph with tree-width at most $t$, one can lift the linear protrusion property from $T$ to $G$.

\section{Protrusions}\label{sec:realprot}

A \emph{$p$-boundaried graph} $P=(A,B)$ is a graph on vertex set 
$A\cup B$, where $A\cap B=\emptyset$, and an enumeration $v_1,\dots,v_p$ of the vertices of $A$. We call $A$ the \emph{boundary} of $P$ and $|B|$ the \emph{size} of $P$. 
We say that two $p$-boundaried graphs $P=(A,B)$ and $P'=(A',B')$, where $A=\{v_1,\dots,v_p\}$ and $A'=\{v'_1,\dots,v'_p\}$, are \emph{compatible} if the function $\phi(v_i)=v'_i$ for all $i=1,\dots,p$ is an isomorphism from $P[A]$ to $P'[A']$.

The \emph{signature} of a subset $X$ of vertices of $P$ is $\sigma(X):=\{i\in [p]:v_i\in X\}$, that is the set of indices of the intersection $X\cap A$. Given some graph $F$ on the set of vertices $[f]:=\{1,\dots ,f\}$ and an induced minor model $\mathcal{M}=\{X_1,\dots ,X_f\}$ of $F$ in $P=(A,B)$, we call \emph{signature} of $\mathcal{M}$ the $f$-tuple of disjoint sets $\sigma(\mathcal{M}):=(\sigma(X_1),\dots ,\sigma(X_f))$.
We say that $(F,\sigma(\mathcal{M}))$  is a \emph{boundaried induced minor} of $P$ and its size is $f$. Since a given graph $F$ can appear in at most $(|A|+1)^{f}$ distinct boundaried induced minors of $P$, the set of all boundaried induced minors of $P$ with size at most $f$, called the \emph{$f$-folio} of $P$, is therefore bounded. To illustrate the notion, note that two $p$-boundaried graphs with the same $p$-folios are compatible, as one can consider the induced minor consisting of the $p$ singleton vertices of their respective boundaries.

A \emph{$p$-protrusion} $P=(A,B)$ in a graph $G$ is a boundaried graph induced by two disjoint sets of vertices: the \emph{boundary} $A$, such that $|A|\leq p$, and the \emph{interior} $B$, such that $N_G(B)\subseteq A$. We will often assume that the boundary is of size exactly $p$. 
Let $A=\{v_1,\dots ,v_p\}$ be the boundary of a $p$-protrusion $P=(A,B)$ of $G$. Given another boundaried graph $P'=(A',B')$ compatible with $P$ (where $A'=\{v'_1,\dots ,v'_p\}$), one can \emph{replace} $P$ by $P'$ in $G$ to form a new graph $G'$ as follows: delete $B$ from $G$, add a disjoint independent copy of $P'$, and identify $A$ with $A'$ by contracting all pairs $\{v_i,v'_i\}$. We then say that $G'$ is obtained by \emph{protrusion replacement} from $G$, or \emph{protrusion reduction} when the size of $B'$ is less than the size of $B$. For simplicity, when $P$ is replaced by $P'$, we will often assume that $A'$ is equal to $A$.

\begin{lemma}\label{lem:kHreduction}
    Let $G$ be a graph with a $p$-protrusion $P=(A,B)$ and $G'$ be obtained by replacing $P$ with $P'=(A,B')$ with the same $(f+p)$-folio. Then $G$ and $G'$ have the same induced minors of size at most $f$.
\end{lemma}

\begin{proof}
We just have to show that if $F$ on vertex set $[f]$ is an induced minor of $G$, then it is an induced minor of $G'$. We consider a minor model $\{Y_1,\dots ,Y_f\}$ of $F$ in $G$ and denote by $X_i$ the intersection of $Y_i$ and $V\setminus (A\cup B)$ for all $i=1,\dots ,f$. We denote by $\mathcal{Z}_i=\{Z_i^1,\dots ,Z_i^{k_i}\}$ the connected components of $G[Y_i\cap (A\cup B)]$. If $Y_i\cap A=\emptyset$, then since $G[Y_i]$ is connected, we have $k_i=0$ ($Y_i$ disjoint from $P$) or $k_i=1$ ($Y_i$ inside $B$). If $Y_i\cap A\neq\emptyset$, note that $k_i\leq |Y_i\cap A|$. In particular we always have $k_i\leq |Y_i\cap A|+1$, and therefore $k_1+\dots +k_f\leq f+p$. Since $P$ and $P'$ have the same $(f+p)$-folios, there exist disjoint corresponding connected subsets $W^{j}_i$ in $G'[A\cup B']$ such that $W_i^j\cap A=Z_i^j\cap A$ for all $j\in [k_i]$. Moreover there is an edge between $W_i^j$ and $W_a^b$ if and only if an edge exists between $Z_i^j$ and $Z_a^b$, for all $j\in [k_i]$ and $b\in [k_a]$.

We now set $Y'_i:=X_i\cup W_i^1\cup \dots \cup W_i^{k_i}$ for all $i=1,\dots ,f$. Since $G[Y_i]$ is connected, $G'[Y'_i]$ is also connected. Also, $G[Y_i\cup Y_j]$ and $G'[Y'_i\cup Y'_j]$ are both connected or both not connected, for all $i,j\in [f]$. In particular, $\{Y'_1,\dots ,Y'_f\}$ forms an induced $F$-minor model in $G'$.
\end{proof}

\begin{corollary}\label{cor:kHreduction}
    Let $G$ be a $kH$-free graph with a $p$-protrusion $P=(A,B)$ and $G'$ be obtained by replacing $P$ with $P'=(A,B')$ with the same $(kh+p)$-folio. Then $G'$ is $kH$-free.
\end{corollary}

\begin{proof}
The property of being $H$-minor free can be expressed as being $H'$-induced minor free for all possible supergraphs $H'$ of $H$ on $h$ vertices. Similarly, the property of being $kH$-free can be expressed as a bounded number of forbidden induced minors on $kh$ vertices (all independent unions of $k$ supergraphs of $H$).
\end{proof}

We now extend Corollary~\ref{cor:kHreduction} in order to perform protrusion replacements which not only preserve $kH$-freeness but also the size of a minimum $H$-hitting set in $G$. The first difficulty here is to encode the fact that one can remove some subset $X$ of at most $s$ vertices. This can be done by considering all possible folios one can obtain by removing such an $X$ (subject to all possible intersections $A\cap X$ with the boundary $A$). The second difficulty is that a $p$-protrusion $P$ can intersect a minimum $H$-hitting set on many vertices. The key-remark is that if $P$ does not contain $H$ as a minor (by this we mean that $P[B]$ is $H$-minor free), then a minimum $H$-hitting set cannot have more than $p$ vertices in $P$, since just picking the boundary would be a better choice. Note that we use the connectivity of $H$ for this.

Given a subset of vertices $X$ of some boundaried graph $P=(A,B)$, where $A$ is $\{v_1,\dots ,v_p\}$, we denote by $P\setminus X=(A\setminus X,B\setminus X)$ the boundaried graph in which $A\setminus X=\{v_1,\dots ,v_p\}\setminus X$ is enumerated in the same order as in $A$. 
The \emph{$(s,f)$-folio set} of $P=(A,B)$ is the set of all possible triples $(r,\sigma(Y),\mathcal{F})$ where $r\leq s$, $Y$ is a subset of vertices of size $r$ in $A\cup B$ and $\mathcal{F}$ is the $f$-folio of $P\setminus Y$.

\begin{lemma} \label{lem:hittingred}
    Let $H$ be a connected graph, and $G$ be a $kH$-free graph with a $p$-protrusion $P=(A,B)$ such that $G[B]$ is $H$-minor free. Let $G'$ be obtained by replacing $P$ by $P'=(A,B')$ with the same $(kh+p)$-folio and the same $(p,h+p)$-folio set. Then $G'$ is $kH$-free and $\tau_H(G)=\tau_H(G')$.
\end{lemma}

\begin{proof} By Corollary~\ref{cor:kHreduction}, the graph $G'$ is also $kH$-free. We only show $\tau_H(G)\leq \tau_H(G')$, the reverse direction follows by swapping the roles of $G$ and $G'$.  Consider a minimum $H$-hitting set $X'$ of $G'$. Since $P$ and $P'$ have the same $h$-folio, $G'[B']$ is also $H$-minor free. Since $H$ is connected, then $X'$  satisfies $|X'\cap (A\cup B')|\leq p$, otherwise $(X'\setminus B')\cup A$ would be a smaller hitting set. Therefore $Y'=X'\cap (A\cup B')$ has size at most $p$. 

Since $P$ and $P'$ have the same $(p,h+p)$-folio set (applied to $Y'$), there exists $Y\subseteq A\cup B$ such that $Y\cap A=Y'\cap A$ and $|Y|=|Y'|$, and such that the boundaried graphs $Q=P\setminus Y$ and $Q'=P'\setminus Y'$ have the same $(h+p)$-folio. We now set $X:=(X'\setminus Y')\cup Y$.
Note that $G\setminus X$ is obtained by replacing the protrusion $Q'$ in 
$G'\setminus X'$ by $Q$.
Since $G'\setminus X'$ is $H$-minor free, by Corollary~\ref{cor:kHreduction} (applied with $k=1$), the graph $G\setminus X$ is also $H$-minor free. We have $\tau_H(G)\leq |X|=|X'|=\tau _H(G')$.
\end{proof}

We need to generalize Lemma~\ref{lem:hittingred} to the parameter $\gamma_H$. We define for this the \emph{$(s,f)$-neighborhood folio} of a boundaried graph $P=(A,B)$ to be the set of all $f$-folios one can obtain by deleting both a set of vertices and the closed neighborhood of a set of vertices of size at most $s$ with a prescribed intersection and neighborhood intersection with $A$. Formally, it is the set of all possible 6-tuples $(a,b,\sigma(X),\sigma(Y),\sigma(N[Y]),\mathcal{F})$ where $a\leq s$, $b\leq s$, $X$ is a subset of vertices of size $a$ in $A\cup B$, $Y$ is a subset of vertices of size $b$ in $A\cup B$ and $\mathcal{F}$ is the $f$-folio of $P\setminus (X\cup N[Y])$. 

\begin{lemma} \label{lem:Nhittingred}
Let $H$ be a connected graph, and $G$ be a $kH$-free graph with a $p$-protrusion $P=(A,B)$ such that $G[B]$ is $H$-minor free, and $G'$ be obtained by replacing $P$ by $P'=(A,B')$ with the same $(kh+p)$-folio and the same $(p,h+p)$-neighborhood folio. Then $G'$ is $kH$-free and $\gamma_H(G)=\gamma_H(G')$.
\end{lemma}

\begin{proof} We follow the proof of Lemma~\ref{lem:hittingred} and show $\gamma_H(G)\leq \gamma_H(G')$.
Assume that $D'$ is a minimum $H$-dominating set of $G'$. Since $P$ and $P'$ have the same $h$-folio, $G'[B']$ is also $H$-minor free, and therefore $Y'=D'\cap (A\cup B')$ has size at most $p$. We denote by $X'$ the set $N(D'\setminus (A\cup B'))\cap A$, which is the set of vertices of $A$ dominated by $D'$ from the outside of $P'$. 

Since $P$ and $P'$ have the same $(p,h+p)$-neighborhood folio (applied to the couple of sets $X',Y'$), there exists $Y\subseteq A\cup B$ and $X\subseteq A\cup B$ such that $X\cap A=X'\cap A$, $Y\cap A=Y'\cap A$, $N[Y]\cap A = N[Y']\cap A$, $|Y|=|Y'|$, $|X|=|X'|$, and such that the boundaried graphs $Q=P\setminus (X\cup N[Y])$ and $Q'=P'\setminus (X'\cup N[Y'])$ have the same $(h+p)$-folio. Note that $X=X'$ since $X'\subseteq A$, $X\cap A=X'\cap A$, and $|X|=|X'|$. 

We set $D:=(D'\setminus Y')\cup Y$. Note that $G\setminus N[D]$ is obtained by replacing the protrusion $Q'$ in 
$G'\setminus N[D']$ by $Q$.
Since $G'\setminus N[D']$ is $H$-minor free, by Corollary~\ref{cor:kHreduction} (applied with $k=1$), the graph $G\setminus N[D]$ is also $H$-minor free. We have $\gamma_H(G)\leq |D|=|D'|=\gamma_H(G')$.
\end{proof}

\section{Linear Protrusion Property}\label{sec:LPP}

Recall that a $p$-protrusion $P=(A,B)$ satisfies $|A|\leq p$.
Two protrusions $P=(A,B)$ and $P'=(A',B')$ in a graph $G$ are \emph{independent} if $B$ and $B'$ are disjoint and independent (no edge between them). An \emph{independent} family of $p$-protrusions of $G$ is a set of pairwise independent $p$-protrusions of $G$. 
A class $\mathcal C$ of graphs has the \emph{$(p,L)$-Linear Protrusion Property} (for short $(p,L)$-LPP) if there exists $c_L>0$ such that every graph $G$ in $\mathcal C$ has an independent family of $p$-protrusions $\cal P$ where $|\mathcal{P}|\geq \lfloor c_L n\rfloor$ and every protrusion  of $\cal P$ has size  at least $L$. Moreover, $\mathcal{C}$ has the \emph{$p$-Strong Linear Protrusion Property} (for short $p$-SLPP) if for every $L$, the class $\mathcal{C}$ has the $(p,L)$-LPP. 

\begin{theorem}\label{thm:SLPPTW}
    Graphs with tree-width $t$ have the $2t$-SLPP.
\end{theorem}

\begin{proof}  
For every integer $L$, we want to show the existence of $c_L>0$ such that every graph $G$ with tree-width at most $t$ has $\lfloor c_L.n\rfloor$ pairwise independent $2t$-protrusions of size at least $L$. Note that we just have to show this property for large enough values of $L$, so we can assume $L\geq 2t+2$.

Consider a tree-decomposition $T$ of $G$ which is rooted at some bag $B_0$ (possibly empty when $G$ is disconnected). Given a bag $B_i$, we denote by $T_i$ the subtree of $T$ rooted at $B_i$ and by $G_i$ the subgraph of $G$ induced by the vertices of the union of the bags of $T_i$. We also denote by $G_i^*$ the induced subgraph $G_i\setminus B_i$. We can assume without loss of generality that all the graphs $G_i$ are connected (apart possibly $G_0$) and that all $B_i$ are pairwise incomparable with respect to inclusion. 

We consider the subtree $T'$ of $T$ consisting of the bags $B_i$ such that $|V(G_i)|\geq L+t+1$. We denote by $b$ the number of nodes of $T'$ and by $\ell $ its number of leaves. We denote by $G'$ the graph induced by the union of the vertices of the bags of $T'$. Recall that graphs with tree-width at most $t$ have linear neighborhood complexity (for some constant $c_t\geq 1$). We now set $c_L:=1/40c_t(t+1)L^3$

We discuss the values of $b$ and $\ell $, with three main cases.

\begin{itemize}
    \item Assume $\ell \geq n/30c_t(t+1)L^3$. Consider a leaf $B_i$ of $T'$. Observe that $N(G_i^*)\subseteq B_i$ and $|V(G_i^*)|\geq |V(G_i)|-|B_i|\geq L$. Thus $N(G_i^*)$ has size at most $t+1\leq 2t$ and therefore $P_i=(N(G_i^*),G_i^*)$ is a $2t$-protrusion of size at least $L$. In particular, by selecting a protrusion for each leaf $B_i$, we have found our independent family of 
    size $\lfloor c_L.n\rfloor$.

    \item Assume that $b\geq n/4c_t(t+1)L^2$ and $\ell<n/30c_t(t+1)L^3$. The number $\ell'$ of nodes of $T'$ with no child in $T'$ (leaves) or more than one child in $T'$ is at most $2\ell$. So $\ell'\leq 2\ell<n/15c_t(t+1)L^3$. A path $P$ of $T'$ (seen as a rooted path going away from the root) is \emph{proper} if all its nodes have one child in $T'$ and $P$ is maximal for this property. To a proper path $P$, one can associate the (private) child bag $B_P$ of $T'$ of the last node of $P$. Note that by maximality, $B_P$ does not have a unique child in $T'$, hence the number of proper paths is at most $\ell'$. Each proper path $P$ of $T'$ can be partitioned into $\lfloor|P|/2L\rfloor$ subpaths of length $2L$, and some leftover set of nodes with size less than $2L$. The collection of all these subpaths of length $2L$ form a family $\mathcal{F}=\{P'_1,\dots ,P'_N\}$. A node of $T'$ is \emph{good} if it belongs to $\bigcup \mathcal{F}$. Each of these paths $P'_i$ of $T'$ starts with a bag $B_i$ which has a parent node $B'_i$ and ends with a bag $B_j$ which has a unique child node $B'_j$ in $T'$. Each $P_i'$ can be seen as a protrusion on vertex set $V(G_i)\setminus V(G'_j)\setminus B_j)$ whose boundary $A_i:=(B_i\cap B'_i)\cup (B_j\cap B'_j)$ has size at most $2t$. The set of vertices of $P_i'$ which are not in $A_i$  has size at least $2L-2t\geq L$. We then just have to show that $N$ is large enough.

    To get a lower bound on $N$, note that at most $2L\ell'$ nodes of proper paths are not used in $\mathcal{F}$ (the leftover of the partition modulo $2L$). The number of good nodes is at least 
    $$b-\ell'-2L\ell'\geq b-3L\ell'\geq n/4c_t(t+1)L^2- n/5c_t(t+1)L^2\geq n/20c_t(t+1)L^2.$$
    
    In particular, $N>n/40c_t(t+1)L^3$ since the good nodes are grouped into subsets of size $2L$. We have found our independent family of 
    size $\lfloor c_L.n\rfloor$.
    
\item Assume now that $b<n/4c_t(t+1)L^2$, in particular $G'$ has at most $n/4c_tL^2$ vertices. We denote by $\mathcal{B}$ the set of connected components of $G\setminus V(G')$. By definition of $T'$, the size of each component in $\mathcal{B}$ is at most $L+t+1$. So the size of $\mathcal{B}$ is at least $(n-n/4c_tL^2)/(L+t+1)>n/2L$ since $L+t+1\leq 3L/2$ and $1-1/4c_tL^2>3/4$. 

Consider now the graph $G/\mathcal{B}$ in which every connected component in $\mathcal{B}$ is contracted to a single vertex. Note that $\bigcup \mathcal{B}/\mathcal{B}$ is an independent set of $G/\mathcal{B}$ denoted by $B$. The neighborhood $A$ of $B$ in $G/\mathcal{B}$  has size at most $n/4c_tL^2$ since it is contained in $V(G')$. Partition the elements of $B$ into $C_1,\dots ,C_r$ according to their neighborhoods in $A$.  By the linear neighborhood complexity property of $G/\mathcal{B}$ (which has tree-width at most $t$), we have $r\leq n/4L^2$. Refine further the partition $(C_i)$ into a partition $(D_i)$ in which each $C_i$ is partitioned into some subsets of size $L$ and a subset of size $|C_i|\mod L$. Denote by $\mathcal{D}$ the subcollection of subsets $D_i$ with size exactly $L$, and set $d:=|\mathcal{D}|$. Each $D_i$ corresponds to a subset of vertices $U_i$ of $G$ obtained by uncontracting each element $b_j$ of $B$ in $D_i$ to the corresponding set $B_j$ of $\mathcal B$. Note that $N(U_i)$ in $G$ has size at most $t$ otherwise $D_i,N(D_i)$ would be a $K_{L,|N(D_i)|}$ subgraph in $G/\mathcal{B}$, and thus $U_i,N(U_i)$ would be a $K_{L,|N(D_i)|}$-minor model in $G$, contradicting the fact that the tree-width of $G$ is at most $t$. So  $\mathcal{P}=\{(N(U_1),U_1),\dots ,(N(U_d),U_d)\}$ is an independent family of  $t$-protrusions each with size at least $L$. 

We just have to obtain a lower bound on $d$. 
The set $B$ has size at least $n/2L$ and the leftover vertices (not appearing in the $D_i$) are at most $Lr\leq Ln/4L^2=n/4L$. So at least $n/4L$ vertices belong to the $U_i$, and since each $U_i$ has size $L$ we have $d\geq n/4L^2$, and our conclusion. 
\end{itemize}
\end{proof}

We can now show the linear protrusion property for $kH$-free graphs with no small $H$-minor model, which is the backbone of both~\Cref{thm:EP} and~\Cref{thm:EP-sparse}.

\begin{restatable}{theorem}{thmLPP}\label{thm:LPPkH}
    There exists an integer $p$ such that for every integer $L$, there exists $g_L$ such that the class of $kH$-free graphs with $H$-girth more than $g_L$ has the $(p,L)$-LPP. 
\end{restatable}

\begin{proof}
An equivalent task is to find $N$, $p$, $c_L$ and $g_L$ such that every $kH$-free graph $G$ with at least $N$ vertices and $H$-girth more than $g_L$ has a family $\mathcal{P}$ of $\lfloor c_Ln\rfloor$ pairwise independent $p$-protrusions of size at least $L$. Indeed, to translate into a statement without a lower bound on $n$, we just have to set a new parameter $c'_L=\min(c_L,1/N)$. The introduction of $N$ gives more flexibility for parameter tuning. Let us recall them:

\begin{itemize}
    \item $m$ is the maximum of $|V(H)|$ and $|E(H)|$. Note that $H$ is therefore a minor of $K_{m,m}$.
    \item  $t$ is the maximum tree-width of an $H$-minor free graph, and $p:=2t+m$.
    \item By Theorem~\ref{thm:DVO}, $kH$-free graphs with no $K_{m,m}$ subgraph have neighborhood complexity at most $c$. Hence $|\{N(v)\cap S:v\in V\}|\leq c.|S|$ for every non-empty $S \subseteq V(G)$. 
    \item By Theorem~\ref{thm:DVO}, $kH$-free graphs with no $K_{m,m}$ subgraph have average degree at most $2\delta$.
    \item By Theorem~\ref{thm:SLPPTW}, there is a constant $\varepsilon >0$ such that every graph $T$ with tree-width at most $t$ has a family of $\varepsilon |V(T)|$ independent $2t$-protrusions with size at least $12ctL^2$.
\end{itemize}

We assume that $G$ has $H$-girth more than $g_L:=\max(32m/\varepsilon,4m/3ctL\varepsilon)$. In particular $G$ does not contain any $K_{m,m}$ subgraph, and therefore its average degree is at most $2\delta$. We also set our parameter $c_L:=\min(\varepsilon /8, ct\varepsilon /2)$.

Consider some integer $d$ (to be chosen later) and partition $G$ into the set $Y$ of vertices of degree at most $d$ and $X:=V\setminus Y$. Observe that the size of $X$ is at most $2\delta n/d$ by maximum average degree.
By Theorem~\ref{thm:KOR}, since the graph $G[Y]$ has maximum degree $d$ and is $kH$-free, it contains a subset $X'$ of size $s'$ only depending on $d$ and $k$ such that $G[Y]\setminus X'$ is $H$-minor free, and thus has tree-width at most $t$.

Set $N=ds'/2\delta$, and thus, since $n\geq N$, we have $s'\leq 2\delta n/d$. We rename both $X\cup X'$ as $X$ and $Y\setminus X'$ as $Y$. 
Note that since $|X| \leq 4\delta n/d$, the ratio $|Y|/|X|$ can be made arbitrarily large by a suitable choice of $d$ and $N$. We then assume from now on that we have both  $\varepsilon |Y|\geq8cm|X|$ and 
$\varepsilon t|Y|\geq|X|$.

By Theorem~\ref{thm:SLPPTW}, there is  a family $\mathcal{P}$ of $8\ell:=\varepsilon |Y|$ independent $2t$-protrusions of $G[Y]$, each of size at least $12ctL^2$. We  distinguish two cases.

\begin{itemize}
    \item There are $4\ell$ protrusions of $\mathcal{P}$ with a connected component of size at least $L$. Let us call $\mathcal{P}'$ the collection of $4 \ell$ protrusions obtained from $\mathcal{P}$ by selecting a single component of size at least $L$ in each. If $2\ell$ protrusions of $\mathcal{P}'$ have neighborhoods in $X$ of size at most $m$, we have our conclusion since they consist of $\varepsilon |Y|/4\geq \varepsilon n/8\geq c_Ln$ many $p$-protrusions of $G$ since $|Y| \geq n/2$. Let us reach a contradiction in the other case. Assume that $2\ell$ protrusions $\mathcal{P}''$ of $\mathcal{P}'$ have neighborhoods of size more than $m$ in $X$. Let $\mathcal{B}'':=\{B_1,B_2,\dots,B_{\ell}\}$ be the interior of the $\ell$ smallest protrusions of $\mathcal{P}''$. Each $B_i$ has size at most $n/\ell$, otherwise the union of $\mathcal{P}'$ would exceed $n$. Thus for each $i$ we have $|B_i|\leq n/\ell=8n/\varepsilon |Y|\leq 16/\varepsilon$ since $|Y|\geq n/2$. Consider the induced graph $G[(\bigcup_{i=1}^{\ell} B_i)\cup X]$ and contract each $B_i$ to a single vertex $b_i$ to form a new graph $G'$. Note that $G'$ is $kH$-free since it is an induced minor of $G$. Note also that $G'$ has no $K_{m,m}$ subgraph since this would be an $H$-minor model with size at most $32m/\varepsilon$. Since $b_i$ has at least $m$ neighbors in $X$ and $\ell>c.m.|X|$, by linear neighborhood complexity (and pigeonhole principle), there are $m$ vertices $b_i$ with the same neighborhood. This gives a $K_{m,m}$ and a contradiction. 
    \item There are $4\ell$ protrusions of $\mathcal{P}$ with no connected component of size at least $L$. Splitting these protrusions into connected components, we form a family $\mathcal{P}'$ of at least $12ctL.4\ell$ connected independent $2t$-protrusions. Note that the boundary $C$ of $\mathcal{P}'$ in $G[Y]$ has size at most $8\ell t$ since each protrusion of $\mathcal{P}$ has boundary size at most $2t$. Since the size of $X$ is at most $8\ell t$, the boundary $C'$ of $\mathcal{P}'$ in $G$ has size at most $16\ell t$. Consider $\mathcal{P}''$ consisting of the smallest $24ctL\ell$ protrusions of $\mathcal{P}'$ and observe that each of their size is at most $n/24ctL\ell=8n/24ctL \varepsilon|Y|\leq 2/3ctL\varepsilon$ since $|Y| \geq n/2$. Contract each protrusion $B_i$ of  $\mathcal{P}''$ to a single vertex $b_i$ and observe that the resulting graph $G'$ is $kH$-free with no $K_{m,m}$ (since this would give an $H$-minor model with size $4m/3ctL\varepsilon$). By linear neighborhood complexity, the number of distinct neighborhoods of the $b_i's$ in $G'$ is at most $c|C'|\leq 16ct\ell $. Now, we greedily group the $b_i's$ with identical neighborhoods into subsets $B'_j$ of size $L$ to form a family of protrusions 
    $\mathcal{S}=\{B'_1,\dots, B'_s\}$. The remaining $b_i's$ which cannot be grouped are at most $L$ times the number of distinct neighborhoods, that is $16ctL\ell$. The number of $b_i's$ is $24ctL\ell$, so at least $8ctL\ell$ of them belong to some $B'_j$, in particular $s\geq 8ct\ell\geq ct\varepsilon |Y|\geq ct\varepsilon n/2\geq c_Ln$. Each $B'_j$ has neighborhood size less than $m$, otherwise $G'$ would contain a $K_{L,m}$ subgraph, which contradicts $K_{m,m}$-freeness (assuming w.l.o.g $L \geq m$). To conclude observe that the (uncontracted) $B_j$'s forms a family of $ct\varepsilon n/2$ many independent $m$-protrusions of size at least $L$. This is our solution.
\end{itemize}

\end{proof}

\section{Induced Erd\H{o}s--P\'osa property for planar minors}\label{sec:EPP}

We begin by showing the Erd\H{o}s--P\'osa property for connected graphs $H$.

\begin{theorem}\label{thm:EP-connected}
For any connected planar graph $H$, there exists a function $f$ such that every $kH$-free graph $G$ satisfies $\gamma_H(G)\leq f(k)$.
\end{theorem}

\begin{proof}

We use our definition for $m$ and $t$ and $p=2t+m$. Let us set $L$ to be such that for every $p$-boundaried graph $P=(A,B)$ of size at least $L$, there exists a $p$-boundaried graph $P'=(A',B')$ with smaller size, same $(kh+p)$-folio, and same $(p,h+p)$-neighborhood folio. Such a bound $L$ exists since there is only a bounded number of different types.

By Theorem~\ref{thm:LPPkH}, there exists $g_L$ and $c_L$ such that every $kH$-free graph $G$ with $H$-girth more than $g_L$ contains $\lfloor c_L n\rfloor$ many $p$-protrusions of size at least $L$.

We show $f(2):=\max (g_L,2/c_L)$. Assume for contradiction that $G$ is a $2H$-free graph with minimum size such that $\gamma_H(G)>f(2)$. In particular $n>2/c_L$. If $G$ contains an $H$-minor model $\mathcal{H}$ of size at most $g_L$, then $G\setminus N[V(\mathcal{H})]$ is $H$-minor free, a contradiction. Otherwise, by Theorem~\ref{thm:LPPkH}, there are $\lfloor c_L n\rfloor>2$ independent $p$-protrusions of size at least $L$. Since $G$ is $2H$-free, one of them, say $P=(A,B)$, is such that $G[B]$ is $H$-minor free. But then there exists a $p$-boundaried graph $P'=(A',B')$ with smaller size, same $(2h+p)$-folio, and same $(p,h+p)$-neighborhood folio.
Thus by Lemma~\ref{lem:Nhittingred}, the (smaller) graph $G'$ obtained by replacing $P$ by $P'$ is $2H$-free and satisfies $\gamma _H(G')=\gamma _H(G)$, a contradiction to the minimality of $G$.

By applying the same argument, we obtain more generally that $f(k)=\max(k/c_L,g_L+f(k-1))$.
\end{proof}

To handle the case where $H$ is disconnected, we need the following straightforward adaptation of \Cref{lem:Nhittingred}, allowing to reduce protrusions when they contain no minor of \textsl{any} connected component of $H$.
\begin{lemma} \label{lem:Nhittingred-disconnected}
Let $H$ be a planar graph, and $G$ be a $kH$-free graph with a $p$-protrusion $P=(A,B)$ such that $G[B]$ is $C$-minor free for every connected component $C$ of $H$.
Then, the graph $G'$ obtained by replacing $P$ by $P'=(A,B')$ with the same $(kh+p)$-folio and the same $(p,h+p)$-neighborhood folio is $kH$-free and $\gamma_H(G)=\gamma_H(G')$.
\end{lemma}
\begin{proof}
    The proof is exactly the same as~\Cref{lem:Nhittingred}. This follows from the fact that any minimum $H$-dominating set $D'$ in $G'$ intersects $A \cup B'$ on at most $p$ elements.  
\end{proof}

With~\Cref{lem:Nhittingred-disconnected} in hand, we are ready to show the induced Erd\H{o}s--P\'osa property for all planar graphs.
\thmEP*
\begin{proof}
    We show the existence of  $f(k,H)$ by induction on $k$ and the number of components $d$ of $H$. The case $d=1$ for any $k$ corresponds to~\Cref{thm:EP-connected}.
    Consider $d \geq 2$ and assume that the result holds for any planar graph $H$ with less than $d$ connected components. Let $H$ be a planar graph with connected components $H_1,...,H_d$. Recall that $h$ denotes the number of vertices of $H$.
    Set now $L$ as well as $g_L,c_L$ as in the proof of~\Cref{thm:EP-connected}, assume without loss of generality $L \geq p$.
    Define $f(k,H) = \max(\max_i f(k,H \setminus H_i) + 8dk/c_L,g_L+f(k-1,H))$.
    Let $G$ be a graph of minimum size $n$ that is $kH$-free, and such that $\gamma_H(G) > f(k,H)$. If $G$ has $H$-girth at most $g_L$, we reduce to the $(k-1)H$-free case, for a bound of $g_L+f(k-1,H)$. Otherwise, we apply~\Cref{thm:LPPkH} to get a set of $\lfloor c_L n \rfloor$ independent $p$-protrusions of size at least $L$. 

    We first deal with the case where at least $c_L n / 2$ of these protrusions contain a minor of some $H_i$ in their interior.
    Then, by the pigeonhole principle, at least $c_L n / 2d$ protrusions contain an $H_1$-minor model, without loss of generality.
    Of those, half of them each have size at most $4d/c_L$.
    If $k \geq c_Ln/4d$, $n$ is at most $4dk/c_L$, so we have our conclusion by taking the whole $V(G)$ in our $H$-dominating set.
    If not, we consider a subset $\{P_i = (A_i,B_i):i=1,\dots,k\}$ of these latter protrusions. Let then $X_1 = \bigcup_{i \leq k} (A_i \cup B_i)$, and note that $|X_1| \leq k(p+4d/c_L) \leq 8kd/c_L$ since $p \leq L \leq 4d/c_L$.
    Observe now that the graph $G \setminus X_1$ is $k(H \setminus H_1)$-free, because $X_1$ contains $k$ independent copies of $H_1$ with no neighbors in $G \setminus X_1$. By induction, it admits a minimum $(H \setminus H_1)$-dominating set $X'$ of size at most $f(k,H \setminus H_1)$. Then, $X = X' \cup X_1$ is a dominating set of $(H \setminus H_1)$-minor models, thus of $H$-minor models. Moreover $|X| \leq f(k,H \setminus H_1) +8dk/c_L$, which concludes this case.
    
    We may now assume that at least $c_L n / 2$ protrusions do not contain a minor of any $H_i$.
    If $n < 2 / c_L$, then the whole set $V(G)$ is a $H$-dominating set of size at most $f(k,H)$ as desired, so we may assume that there exists at least one such protrusion $P = (A,B)$.
    By definition of $L$, there is a $p$-boundaried graph $P'=(A',B')$ with smaller size, same $(kh+p)$-folio, and same $(p,h+p)$-neighborhood folio.
    Then, replacing $P$ with $P'$ yields by~\Cref{lem:Nhittingred-disconnected} a graph $G'$ that is still $kH$-free and satisfies $\gamma_H(G') = \gamma_H(G)$, contradicting the minimality of $G$.
\end{proof}

\section{Logarithmic size hitting set and Maximum Independent Set}\label{sec:MIS}

We now prove the existence of a logarithmic size hitting set for sparse $kH$-free graphs. We first need a sparsification lemma which asserts that in every sparse $kH$-free graph $G$, there is a bounded number of vertices which intersects all small $H$-minor models.

\begin{lemma}\label{lem:sparsification}
For every $g$ and $r$, there exists $h(g,r)$ such that every $kH$-free graph $G$ on $n$ vertices with no $K_{r,r}$ subgraph has a set of vertices $X$ with size at most $h(g,r)$ such that $G\setminus X$ has $H$-girth more than $g$. 
\end{lemma}

\begin{proof}
   Consider a $kH$-free graph $G$ with no $K_{r,r}$ subgraph and a disjoint union $\mathcal{H}_1,\dots ,\mathcal{H}_N$ of $H$-minor models in $G$ each with size at most $g$. We assume also that $N$ is as large as possible. Call $G'$ the subgraph induced by $\bigcup \mathcal{H}_i$. The graph $G''$ obtained by contracting in $G'$ each $\mathcal{H}_i$ to a single vertex has maximum  independent set at most $k-1$. By Tur\'an's theorem, $G''$ has at least $(k-1){\lfloor\frac{N}{k-1}\rfloor\choose 2}$ edges. So the number of edges of $G'$ is $\Omega(N^2/k)$ and its number of vertices is at most $gN$. Therefore $G'$ has a quadratic number of edges. Since graphs without $K_{r,r}$ subgraphs have average degree $O(n^{1-1/r})$, the number of vertices of $G'$ is bounded by some value $h(g,r)$. By maximality of $N$, deleting $V(G')$ from $G$ leaves a graph with $H$-girth more than $g$.
\end{proof}

\subsection{Logarithmic hitting set}

\begin{theorem}\label{thm:EP-sparse-connected}
For any connected planar graph $H$, every $kH$-free graph $G$ on $n$ vertices with no $K_{r,r}$ subgraph satisfies $\tau_H(G)=O_{r,k}(\log n)$.
\end{theorem}
\begin{proof}
Again $p=m+2t$. Let us set $L$ to be such that for every $p$-boundaried graph $P=(A,B)$ of size at least $L$ there exists a $p$-boundaried graph $P'=(A',B')$ with smaller size, same $(kh+4p)$-folio, and same $(p,h+p)$-folio set. Such a bound $L$ exists since there is only a bounded number of different types.

By Theorem~\ref{thm:LPPkH}, there exists $g_L$ and $c_L$ such that every $kH$-free graph $G$ with $H$-girth more than $g_L$ contains $\lfloor c_L n\rfloor$ many independent $p$-protrusions of size at least $L$.

We now proceed by induction on $n$ and show the existence of a constant $C$ such that every $kH$-free graph on $n$ vertices without $K_{2p,2p}$ subgraph satisfies $\tau_H(G)\leq C\log n$.

By Lemma~\ref{lem:sparsification}, there exists $X\subseteq V(G)$ with size at most $h(g_L,2p)$ such that $G_1:=G\setminus X$ has $H$-girth more than $g_L$. We denote by $n_1$ the number of vertices of $G_1$.
By Theorem~\ref{thm:LPPkH}, there are $\lfloor c_L n_1\rfloor$ independent $p$-protrusions in $G_1$ of size at least $L$. At most $k$ of these protrusions contain an $H$-minor model. If $n_1\leq 2k/c_L$, the number of vertices of $G$ is at most $2k/c_L+h(g_L,2p)$, and we set $C$ such that $C\log n\geq 2k/c_L+h(g_L,2p)$. Otherwise, by the choice of $L$ we can reduce  all the (at least $c_Ln_1/2$ many) protrusions without $H$-minor model to smaller ones with same $(kh+4p)$-folio and same $(p,h+p)$-folio set. This (iterated or parallel) sequence of reductions terminates on a graph $G_2$, which by  Lemma~\ref{lem:hittingred} is still $kH$-free and satisfies $\tau_H(G_2)=\tau_H(G_1)$.

Observe that every protrusion $P=(A,B)$ which is reduced (to some $P'=(A,B')$) is such that $B$ does not contain any $H$-minor model, and thus in particular $G_1[A\cup B]$ does not contain a $K_{2p,2p}$-minor model. Since $P$ and $P'$ have the same $(kh+4p)$-folio, $G_2[A\cup B]$ does not contain a $K_{2p,2p}$ subgraph. Also, $G_1$ and $G_2$ coincide outside of the protrusions, and a $K_{2p,2p}$ cannot cross the boundary of a protrusion, so the graph $G_2$ has no $K_{2p,2p}$-subgraph. Therefore by induction, $\tau_H(G_1)=\tau_H(G_2)\leq C\log(|V(G_2)|)$.

Since $|V(G_2)|\leq (1-c_L/2)n_1$, we have $\tau_H(G_1)\leq C\log(n_1)+C\log(1-c_L/2)$. Therefore $\tau_H(G)\leq h(g_L,2p)+ C\log(n_1)+C\log(1-c_L/2)$. We just have to choose $C$ large enough such that $h(g_L,2p)+C\log(1-c_L/2)\leq 0$ to satisfy the induction step.

For the general case, in which $G$ has no $K_{r,r}$ subgraph, we apply Lemma~\ref{lem:sparsification} to find a set $X$ of size at most $h(4p,r)$ such that $G\setminus X$ has no $K_{2p,2p}$ subgraph. In particular $\tau_H(G)\leq h(4p,r)+C\log n$.
\end{proof}

We now move to the proof for all planar graphs $H$, where we will use the generalized protrusion reduction property given by~\Cref{lem:Nhittingred-disconnected}.
\thmLOG*
\begin{proof}
   The argument is similar to the proof of Theorem~\ref{thm:EP}. Either some connected component $H_1$ of $H$ appears in a linear number of $p$-protrusions, and therefore one can find $k$ independent copies of $H_1$ in a bounded size set $X_1$, allowing to reduce to the $k(H\setminus H_1)$-free case. Or there is a linear number of $p$-protrusions which do not contain any component of $H$, and we can reduce all of them, thus reducing the size by some fixed factor. 
\end{proof}

\subsection{Quasi-polynomial-time algorithm for Maximum Independent Set}

Using~\Cref{thm:EP-sparse}, we are now ready to show our algorithm for MIS, which follows the lines of~\cite{BBDEGHTW23}.
\thmMIS*
\begin{proof}
By Theorem~\ref{thm:EP-sparse}, we just have to reduce the problem from the general case to the sparse case. Indeed, if an $H$-hitting set $X$ has logarithmic size, we can guess it in QP-time, and then guess its intersection $I$ with a MIS in polynomial time. Since $G\setminus (X \cup N[I])$ is $H$-minor free, its tree-width is at most $t$, hence one can compute a maximum independent set $I'$ in linear time. We finally output the maximum of $|I|+|I'|$ for all choices of $I$ and $X$. 

Let us now reduce to the sparse case. We construct the graph $G^{2m}$ whose vertices are all the (possibly pairwise intersecting) copies of $K_{m,m}$ subgraphs of $G$. There is an edge $K,K'$ in $G^{2m}$ if there is an edge connecting a vertex of $K$ to a vertex of $K'$ (which is the case when $K\cap K'\neq \emptyset$). Observe that $G^{2m}$ has no independent set of size $k$, so there is a vertex $K$  with degree at least $|V(G^{2m})|/k$. This vertex $K$ corresponds to a copy of 
$K_{m,m}$ in $G$, so there is a vertex $v$ of this copy which is adjacent to at least a $1/2km$ fraction of the vertices of $G^{2m}$.

Note that $v$ can be found in polynomial time $O(n^{4m})$, which is the cost of constructing $G^{2m}$. Now, if one starts a branching process on $v$ to compute MIS in $G$, the branch $v\notin MIS$ recurses on $G\setminus v$ and the branch $v\in MIS$ recurses on $G\setminus N[v]$. In the latter case the number of $K_{m,m}$ decreases by a $1/2km$ fraction. This branching process ends when the current input graph has no $K_{m,m}$ subgraph. Since one branch decreases by 1 and the other decreases by a fraction (of $n^{2m}$), the leaves in this recursion tree correspond to $0,1$ words of length at most $n$ in which the number of 1 is logarithmic. Hence this preprocess reduces $G$ to a quasi-polynomial $n^{O(\log n)}$ number of instances without $K_{m,m}$ subgraphs, in which we can apply the algorithm in the sparse case.  
\end{proof}

\bibliographystyle{abbrv}
\bibliography{ref} 
\end{document}